\documentclass{article}

\usepackage{a4wide}

\usepackage{amsmath}
\usepackage{amssymb}
\usepackage{amsthm}
\usepackage{mathtools}
\numberwithin{equation}{section}

\newtheorem{thm}{Theorem}
\newtheorem{rem}{Remark}
\newtheorem{lem}[thm]{Lemma}

\usepackage[latin1]{inputenc}

\usepackage{hyperref}
\hypersetup{plainpages=false, linktocpage=true, colorlinks=true, breaklinks=true, linkcolor=black, menucolor=black, urlcolor=black, citecolor=black}
\usepackage{graphicx}
\usepackage{floatflt}
\usepackage{subfigure}
\graphicspath{{images/}}
\usepackage[all]{hypcap}

\usepackage{booktabs}
\usepackage{multirow}
\usepackage{placeins}

\newcommand{\R}{ \ensuremath{\mathbb{R}} }
\newcommand{\N}{ \ensuremath{\mathbb{N}} }

\newcommand{\eps}{\varepsilon}
\newcommand{\E}{\mathrm{e}}
\newcommand{\D}{\mathrm{d}}

\title{A convection-diffusion problem with a small variable diffusion coefficient}
\author{Hans-G. Roos and Martin Schopf}

\date{\today}

\begin{document}

\maketitle

\begin{abstract} Consider a singularly perturbed convection-diffusion problem with a small, variable
diffusion. Based on certain a priori estimates for the solution we prove robustness of a finite element
method on a Duran-Shishkin mesh.
\end{abstract}



\begin{section}{Introduction}
Consider the one dimensional boundary value problem
\begin{gather}
   \begin{split}
		\mathcal{L}_\eps u \coloneqq -(\eps u')' - b u' + c u &= f \quad \text{in $(0,1)$},\\
		u(0) &= 0,\\
		u(1) &= 0,
	\end{split}
	\label{eq:prob}
\end{gather}
with smooth functions $\eps,b,c,f:[0,1]\to\R$, satisfying
\begin{gather}
	\begin{aligned}
		0 &< \beta < b(x)\\
		0 &< \underline \eps \le \eps(x) \le \overline \eps \ll 1
	\end{aligned} \quad \text{for $x \in [0,1]$}.
	\label{eq:bandeps}
\end{gather}
Moreover we assume
\begin{gather}
	c \ge 0, \quad  c+ b'/2\ge \gamma>0,\label{new}
\end{gather}
which can be ensured using the assumptions \eqref{eq:bandeps} and the transformation $u = \hat u \E^{\delta x}$ with suitably chosen constant $\delta$, see, for instance, \cite{RST96}.

We do not know any results concerning robust numerical methods for such problems, the only exceptions
are \cite{FO10,FO11}, where $\eps(x)$ has piecewise the special form $\eps_ip_i(x)$ in $\Omega_i$ with different
parameters $\eps_i$.

Assuming additionally $\eps'>-\beta$, we have an outflow boundary layer at $x=0$. It is relatively
technical to prove a priori estimates for derivatives of $u$ and to prove the existence of a solution
decomposition into a smooth part and a layer part. But this can be done with well known techniques
(Kellogg/Tsan; use of extended domains), see the Appendix.

 Under additional assumptions
($\eps'$  is nonnegative and bounded; moreover conditions on $\eps''$, see Theorem 10 and Remark 4) we have: There exists a solution
decomposition
\[
   u=S+E
\]
with
\begin{subequations}\label{dec}
\begin{align}
			|S^{(k)}(x)| & \le C\quad \text{for $k=0,1,2$}, \label{eq:boundS01*}\\
			\shortintertext{and}
			\big |E^{(k)}(x) \big | & \le C \frac{1}{\eps(x)^k} \E^{-\beta e(x)} \quad \text{for $k=0,1,2$}, \label{eq:boundE01*}
		\end{align}
	\end{subequations}
here
\[
  e(x)=\int_0^x\frac{1}{\eps(t)}dt.
\]

Based on the solution decomposition we are going to analyze the finite element method on a special mesh. The weak formulation of the problem uses the bilinear form
\begin{equation}
a(v,w):=(\eps v',w')-(bv',w)+(cv,w).
\end{equation}
Let $V_h\in H_0^1(0,1)$ be the space of linear finite elements. We look for $u_h\in V_h$ such that
\begin{equation}
a(u_h,v_h)=(f,v_h)\quad {\rm for\quad all}\quad v_h\in V_h.
\end{equation}
Define an energy norm by
\[
  \|v\|_{\eps}^2:=\|\eps^{1/2}v'\|_0^2+\|v\|_0^2.
\]
Then we ask: {\it on which layer adapted mesh  can we prove an (almost) robust error estimate for our
finite element method in that energy norm?}

\end{section}
\begin{section}{The mesh and the interpolation error}
Near the layer we use a fine graded mesh, otherwise an equidistant mesh with the step size $h$.
First we introduce a point $\tau^*$ satisfying
\begin{gather}
	e(\tau^*) =- \frac{2}{\beta} \ln h.
	\label{eq:tau}
\end{gather}
Observe that as $e(0) = 0$ and $e$ is strictly increasing \eqref{eq:tau} has a unique solution.

Since $e$ is strictly increasing the choice \eqref{eq:tau} also implies
\begin{gather}
	\E^{-\beta e(x)} \le \E^{-\beta e(\tau^*)} \le h^{2} \qquad \text{for $x \ge \tau^*$}.
	\label{eq:motiv_tau}
\end{gather}
Moreover, $\tau^*$ satisfies
\begin{equation}
- \frac{2}{\beta}\underline \eps \ln h\le \tau^*\le - \frac{2}{\beta}\overline \eps \ln h.
\label{tau}
\end{equation}
\end{section}
Following \cite{DL06}, we introduce near $x=0$ the graded mesh (D-L mesh)
\begin{gather}\label{eq:durangrid}
	\left \{ \begin{aligned}
		x_0 &= 0,\\
		x_1 &= h\delta\underline \eps	,\\
		x_{i+1} &= x_i + h x_i, \quad \text{for } 1\le i \le N^*.
		\end{aligned} \right.
\end{gather}
We choose $N^*$ in such a way that $\tau=x_{N^*+1}$ is the first point with
$\tau\ge \tau^*$. Then, $\tau$ has similar properties as $\tau^*$. In the subinterval
$[\tau,1]$ we use an equidistant mesh with a mesh size of order $O(h)$.

To simplify the notation, we introduce the symbol $\preceq$ and note $A\preceq B$, if there exists a constant
$C$ independent of $\eps$, such that $A\le C B$.

 Because the smooth part $S$ satisfies
$|S''|\le C$, we have for the interpolation error of the piecewise linear interpolant
\[
\|S-S^I\|_0\preceq h^2,\quad\quad |S-S^I|_1\preceq h.
\]
On $[\tau,1]$ we obtain for the layer component
\[
\|E-E^I\|_{0,[\tau,1]}\preceq \|E\|_{\infty,[\tau,1]}\preceq h^2.
\]
Moreover, by an inverse inequality
\begin{equation}
\|\eps^{1/2}(E-E^I)'\|^2_{0,[\tau,1]}\preceq \int_\tau^1 \frac{1}{\eps(x)}e^{-2\beta e(x)}+
     \frac{1}{h^2}\|\eps^{1/2}E^I\|^2_{0,[\tau,1]}\preceq h^2.
\end{equation}

Next we study the interpolation error on the fine subinterval $[0,\tau]$, using the definition of the mesh,
the estimate of $E''$ and $x\le e(x)\eps(x)$:
\begin{gather}\label{}
	\begin{split}
	\|\eps^{-1/2}(E-E^I)\|^2_{0,[0,\tau]}	 &=\int_0^{x_1}\eps^{-1}(E-E^I)^2+
                    \sum_1^{N^*}\int_{x_i}^{x_{i+1}}\eps^{-1}(E-E^I)^2 \\
		& \preceq h^4+\sum_1^{N^*}\int_{x_i}^{x_{i+1}}\eps^{-5}(hx_i)^4e^{-2\beta e(x)}\\
        & \preceq h^4+h^4\int_0^\tau\eps^{-5}x^4e^{-2\beta e(x)}\le h^4\left(1+\int_0^\tau\eps^{-1}(e(x))^4e^{-2\beta e(x)}\right)\\
        & \preceq h^4\left( 1+\int_0^\infty s^4e^{-2\beta s )}\right)\preceq h^4.
	\end{split}
\end{gather}
Thus we obtain
\begin{equation}\label{int}
\|\eps^{-1/2}(E-E^I)\|_{0,[0,\tau]}\preceq h^2 \quad {\rm and}\quad
\|E-E^I\|_{0,[0,\tau]}\preceq h^2\|\eps^{1/2}\|_{0,[0,\tau]}.
\end{equation}
Similarly we get
\begin{gather}\label{}
	\begin{split}
	\|\eps^{1/2}(E-E^I)'\|^2_{0,[0,\tau]}	 &=\int_0^{x_1}\eps((E-E^I)')^2+\sum_1^{N^*}\int_{x_i}^{x_{i+1}}\eps(E-E^I)')^2  \\
		& \preceq h^2\left(1+\int_0^\infty s^2e^{-2\beta s}\right),
	\end{split}
\end{gather}
resulting in
\begin{equation}
\|\eps^{1/2}(E-E^I)'\|_{0,[0,\tau]}\preceq h .
\end{equation}

\begin{section}{The discretization error}
So far we proved $\|u-u^I\|_\eps\preceq h$ and start now to estimate $\|u_h-u^I\|_\eps$. As usual, we have
based on the coercivity of our bilinear form in the given norm
\begin{gather}\label{}
	\begin{split}
	\|u^I-u_h\|^2_\eps	 &\preceq a(u^I-u_h,u^I-u_h) \\
		&=a(u^I-u,u^I-u_h) \\
        &=(\eps (u^I-u)',v_h')-(b(u^I-u)',v_h)+(c(u^I-u),v_h)
	\end{split}
\end{gather}
with $v_h=u^I-u_h$. The first and the third term can be easily estimated, only the convection term with respect
to the layer part $E$ needs some care. We use integration by parts and on the fine part of the mesh
\[
|(E-E^I,(v_h)')|\le \|\eps^{-1/2}(E-E^I)\|_0 \|v_h\|_\eps,
\]
while on the coarse part  an inverse inequality yields
\[
|(E-E^I,(v_h)')|\preceq \frac{1}{h}\|E-E^I\|_0\|v_h\|_0\le \frac{1}{h}\|E-E^I\|_0\|v_h\|_\eps.
\]
Using \eqref{int}, we get finally
\begin{thm}
If there exists a solution decomposition with the properties \eqref{dec}, then the finite element
approximation with linear elements on our DL-Shishkin mesh satisfies
\begin{equation}
\|u_h-u\|_\eps \preceq h .
\end{equation}
\end{thm}
Remark that our result is not fully robust: the number of mesh points used is of order
$O(\psi(\eps,h)\frac{1}{h})$, where $\psi(\eps,h)$ can be estimate by
   $\ln((\overline\eps)/(\underline\eps))+\ln((-\ln h)/h)$.

\end{section}

\begin{section}{Appendix}
Consider  the one dimensional boundary value problem
\begin{gather}
   \begin{split}
		\mathcal{L}_\eps u\coloneqq -(\eps u')' - b u' + c u &= f \quad \text{in $(0,1)$},\\
		u(0) &= 0,\\
		u(1) &= 0.
	\end{split}
	\label{eq41}
\end{gather}
Assume \eqref{eq:bandeps} and \eqref{new}.

The differential equation in \eqref{eq41} can be rewritten in the equivalent form
\begin{gather}
	-\eps u'' - (b + \eps') u' + c u = f.
	\label{eq:prob2}
\end{gather}
Thus the first derivative of $\eps$ has a crucial influence on the behavior of the exact solution: If for instance $\eps' < -b$ then the outflow boundary will shift to the point $x=1$ leading to the formation of an exponential boundary layer at that point. We shall consider the case $\eps' > -\beta \ge -b$ leaving the outflow boundary point at the origin of the unit interval.

\begin{lem}\label{lem:trafoabsch}
	Let $u$ be the solution of \eqref{eq41} and $T$ be the coordinate transformation
	\begin{gather}
		\xi = T(x) = \int_0^x \sqrt{\frac{\underline \eps}{\eps(t)}} \D t,
		\label{eq:trafo}
	\end{gather}
	mapping the domain $(0,1)$ to $\big(0,T(1) \big)$. Then in the transformed variable $\xi$ it holds
	\begin{gather}
	\label{eq:trafoestimuk}
		|\tilde u^{(k)}(\xi)| \le C \left( 1+ \underline \eps^{-k} \E^{-\frac{\sigma + 2 \beta}{2 \sqrt{\underline \eps \overline \eps}}\xi}\right)
	\end{gather}
	with $\tilde u \coloneqq u \circ T^{-1}$ and $\sigma \coloneqq \min_{z \in [0,1]} \eps'(z) > -\beta$.
\end{lem}
\begin{proof}
	Let $T$ be the coordinate transformation defined by \eqref{eq:trafo}.
	As strict monotone mapping $T$ is injective and therefore $T^{-1}:[0,T(1)]\rightarrow [0,1]$, $\xi \mapsto x$ exists.
	The chain rule yields for $\tilde u \big(T(x) \big) = u(x)$ and $x \in (0,1)$:
	\begin{align*}
		u'(x) &= \frac{\D}{\D x} \tilde u \big(T(x) \big) = \tilde u'\big(T(x)\big) T'(x),\\
		u''(x) &= \tilde u''\big( T(x) \big) \big( T'(x) \big)^2 + \tilde u'\big( T(x) \big) T''(x).
	\end{align*}
	Thus the differential equation \eqref{eq:prob2} is transformed into
	\begin{gather*}
		\underline \eps \tilde u''(T(x)) - \left(\big(b(x)+\eps'(x)\big)\sqrt{\frac{\underline \eps}{\eps(x)}}-\frac{\eps(x) \underline \eps \eps'(x)}{2 \sqrt{\frac{\underline \eps}{\eps(x)}}\big(\eps(x)\big)^2}\right) \tilde u'(T(x)) + c(x) \tilde u(T(x)) = f(x).
	\end{gather*}
	Note that the coefficient of the highest derivative of $\tilde u$ is the constant $\underline \eps$ and that the functions $\tilde c \coloneqq c \circ T^{-1}$ and $\tilde f \coloneqq f \circ T^{-1}$ remain bounded. Therefore rewriting \eqref{eq:prob2} in the new variable $\xi$ yields:
	\begin{gather*}
		\underline \eps \tilde u''(\xi) - \tilde b(\xi) \tilde u'(\xi) + \tilde c(\xi) \tilde u(\xi) = \tilde f(\xi),\quad \xi \in [0,T(1)]\\
		\tilde b = \frac{2 b \circ T^{-1} + \eps' \circ T^{-1}}{2} \sqrt{\frac{\underline \eps}{\eps \circ T^{-1}}}.
	\end{gather*}
	In order to obtain bounds on the derivatives of $\tilde u$ we need an estimate $\tilde \beta \le \tilde b(\xi)$ for $\xi \in [0,T(1)]$ with a constant $\tilde \beta > 0$. Equivalently, we provide an estimate $\tilde \beta \le \tilde b\big(T(x)\big)$ for $x \in [0,1]$:
	\begin{align*}
		\tilde b\big(T(x)\big) &= \frac{2 b(x) + \eps'(x)}{2} \sqrt{\frac{\underline \eps}{\eps(x)}} \ge \sqrt{\frac{\underline \eps}{\eps(x)}} \frac{\sigma + 2 \beta}{2}\\
		& \ge \sqrt{\frac{\,\underline \eps\,}{\overline \eps}} \frac{\sigma + 2 \beta}{2} \eqqcolon \tilde \beta > 0.
	\end{align*}
	Remark that $\underline \eps /\eps(t) \le 1$ implies $T(x) \le x$ and hence $T(1) \le 1$:
	\begin{gather*}
		T(x) = \int_0^x \sqrt{\frac{\underline \eps}{\eps(t)}} \D t \le \int_0^x \D t = x.
	\end{gather*}
	Thus, we can apply well known a-priory estimates for the case when $\eps$ is a constant to obtain
	\begin{gather*}
		|\tilde u^{(k)}(\xi)| \le C \left( 1+ \underline \eps^{-k} \E^{-\frac{\tilde \beta}{\underline \eps}\xi}\right) \le C \left( 1+ \underline \eps^{-k} \E^{-\frac{\sigma + 2 \beta}{2 \sqrt{\underline \eps \overline \eps}}\xi}\right) \qedhere
	\end{gather*}
\end{proof}

\begin{lem}\label{lem:absch}
	Set
	\begin{gather}
		\tilde e(x) \coloneqq \int_0^x \frac{1}{\sqrt{\overline \eps \eps(t)}} \D t
	\end{gather}
	The solution $u$ of Problem \eqref{eq41} satisfies
	\begin{subequations}\label{eq:estimu}
		\begin{align}
			|u(x)| &\le C \left( 1+ \E^{-\frac{\sigma + 2 \beta}{2} \tilde e(x)}\right) \label{eq:estimu0},\\
			|u'(x)| &\le C \sqrt{\frac{\overline \eps}{\eps(x)}} \left( 1+ \underline \eps^{-1} \E^{-\frac{\sigma + 2 \beta}{2} \tilde e(x)}\right) \label{eq:estimu1},\\
			|u''(x)| &\le C \left( \frac{\overline \eps}{\eps(x)} \left( 1+ \underline \eps^{-2} \E^{-\frac{\sigma + 2 \beta}{2} \tilde e(x)}\right) + \frac{\overline \eps \eps'(x)}{2 \big( \eps(x)\big)^2} \left( 1+ \underline \eps^{-1} \E^{-\frac{\sigma + 2 \beta}{2} \tilde e(x)}\right) \right) \label{eq:estimu2}.
		\end{align}
	\end{subequations}
	with $\sigma \coloneqq \min_{z \in [0,1]} \eps'(z) > -\beta$.
\end{lem}
\begin{proof}
	Lemma \ref{lem:trafoabsch} yields
	\begin{align*}
		\left|(u \circ T^{-1})^{(k)}(\xi) \right| = |\tilde u^{(k)}(\xi)| \le C \left( 1+ \underline \eps^{-k} \E^{-\frac{\sigma + 2 \beta}{2 \sqrt{\underline \eps \overline \eps}}\xi}\right).
	\end{align*}
	The transformation $\xi = T(x)$ gives
	\begin{gather}
		\left|(u \circ T^{-1})^{(k)}\big(T(x)\big) \right| \le C \left( 1+ \underline \eps^{-k} \E^{-\frac{\sigma + 2 \beta}{2 \sqrt{\underline \eps \overline \eps}} T(x)}\right) = C \left( 1+ \underline \eps^{-k} \E^{-\frac{\sigma + 2 \beta}{2} \tilde e(x)}\right).
		\label{eq:estimuk}
	\end{gather}
	We use \eqref{eq:estimuk} to deduce our proposition. First \eqref{eq:estimu0} is an immediate consequence of \eqref{eq:estimuk} for $k=0$. Next we want to verify \eqref{eq:estimu1}. A simple calculation yields
	\begin{gather*}
		(u \circ T^{-1})'(\xi) = u'\big(T^{-1}(\xi)\big)\left( T^{-1} \right)'(\xi) = u'\big(T^{-1}(\xi)\big) \frac{1}{T'\big(T^{-1}(\xi)\big)}.		
	\end{gather*}
	With $\xi=T(x)$ we conclude
	\begin{gather}
		(u \circ T^{-1})'\big(T(x)\big) = u'(x) \frac{1}{T'(x)}.
		\label{eq:estimu1lhs}
	\end{gather}
	Collecting \eqref{eq:estimuk} with $k=1$ and \eqref{eq:estimu1lhs} the estimate \eqref{eq:estimu1} follows.
	Same techniques yield
	\begin{gather}
		\begin{split}
			\left|(u \circ T^{-1})''\big(T(x)\big) \right| &= \left|u''(x)\frac{1}{\big(T'(x)\big)^2} - u'(x) \frac{T''(x)}{\big(T'(x)\big)^2} \right| \\
			&\ge |u''(x)|\frac{1}{\big(T'(x)\big)^2} - |u'(x)| \frac{|T''(x)|}{\big(T'(x)\big)^2}
		\end{split}
		\label{eq:estimu2lhs}
	\end{gather}
	Combining \eqref{eq:estimuk} with $k=2$ and \eqref{eq:estimu2lhs} we obtain
	\begin{align*}
		|u''(x)| &\le C \big(T'(x)\big)^2 \left( 1+ \underline \eps^{-2} \E^{-\frac{\sigma + 2 \beta}{2}\tilde e(x)}\right) + |u'(x)| |T''(x)|\\
		&\le C \frac{\overline \eps}{\eps(x)} \left( 1+ \underline \eps^{-2} \E^{-\frac{\sigma + 2 \beta}{2} \tilde e(x)}\right)+ |u'(x)| \frac{\overline \eps \eps'(x)}{2\sqrt{\frac{\overline \eps}{\eps(x)}}\big(\eps(x)\big)^2}.
	\end{align*}
	Using \eqref{eq:estimuk} with $k=1$ for the second term the proof is complete.
\end{proof}

\begin{rem}
	In the classical constant setting $\eps \equiv \underline \eps = \overline \eps$ the formulas \eqref{eq:estimu} reduce to the well-known form
	\begin{gather}
		\left| u^{(k)}(x) \right| \le C \left( 1+ \eps^{-k} \E^{-\frac{\beta}{\eps}x}\right), \quad k=0,1,2.
		\label{eq:classic}
	\end{gather}
\end{rem}


Unfortunately, all summands of the right hand side of the bounds \eqref{eq:estimu1} and \eqref{eq:estimu2} have a large multiplier if $\eps$ changes on a huge scale. Moreover the exponential decay in the estimates \eqref{eq:estimu} appears to be suboptimal. In order to provide better bounds we will use the following Lemmas.

\begin{lem}\label{lem:maxprinciple}
	The differential operator $\mathcal{L}_\eps$ obeys the following maximum principle: For any function $v \in C^2(a,b) \cap C[a,b]$
	\begin{gather*}
		\left.
		\begin{aligned}
			\mathcal{L}_\eps v &\le 0 \quad {\rm in}\,\, (a,b),\\
			v(a) &\le 0,\\
			v(b) &\le 0
		\end{aligned}
		\right \} \quad \Longrightarrow \quad v \le 0 \quad \text{on $[a,b]$.}
	\end{gather*}
\end{lem}
\begin{proof}
	A proof can be found e.g.~in \cite{PW84}.
\end{proof}

	The maximum principle applied to $v_1-v_2$ also yields a comparison principle.

\begin{lem}\label{lem:intestim}
	Let $a < x$, $\ell \in \N_0$, suppose $\eps' \ge \sigma_0 \ge 0$ on $[a,x]$ and set $e_a(t) \coloneqq \int_a^t 1/\eps(z) \D z$. Then
		\begin{gather}
			\int_a^x \eps(t)^\ell \E^{\gamma e_a(t)} \D t \le \frac{1}{\gamma + (\ell + 1)\sigma_0} \big( \eps(x)^{\ell + 1} \E^{\gamma e_a(x)} - \eps(a)^{\ell + 1} \big).
			\label{eq:exp0absch}
		\end{gather}
	for $\gamma > - (\ell + 1)\sigma_0$.
\end{lem}
\begin{proof}
	Since $\sigma_0 \le \eps'(t)$ for $t \in [a,x]$ multiplication with $\eps(t)^\ell \exp \big(\gamma e_a(t) \big) > 0$ and integration yields
	\begin{gather*}
		\sigma_0 \int_a^x \eps(t)^\ell \E^{\gamma e_a(t)} \D t \le \int_a^x \E^{\gamma e_a(t)} \eps(t)^\ell \eps'(t) \D t.
	\end{gather*}
	Integration by parts gives
	\begin{gather}
		\sigma_0 \int_a^x \eps(t)^\ell \E^{\gamma e_a(t)} \D t \le \frac{1}{\ell+1} \left( \eps(x)^{\ell+1} \E^{\gamma e_a(x)} - \eps(a)^{\ell+1} - \int_a^x \eps(t)^{\ell+1} \gamma \E^{\gamma e_a(t)} e_a'(t) \D t \right).
		\label{eq:afterpi}
	\end{gather}
	Inserting
	\begin{gather*}
		- \int_a^x \eps(t)^{\ell+1} \gamma \E^{\gamma e_a(t)} e_a'(t) \D t = - \gamma \int_a^x  \eps(t)^\ell \E^{\gamma e_a(t)} \D t
	\end{gather*}
	into \eqref{eq:afterpi} we obtain
	\begin{gather*}
		\left(\frac{\gamma}{\ell + 1} + \sigma_0\right) \int_a^x \eps(t)^\ell \E^{\gamma e_a(t)} \D t \le \frac{1}{\ell+1} \left( \eps(x)^{\ell+1} \E^{\gamma e_a(x)} - \eps(a)^{\ell+1} \right)
	\end{gather*}
	and \eqref{eq:exp0absch} follows.
\end{proof}

Next, we want to proof some pointwise bounds for the solution of the following problem in a possibly extended domain $(a,1)$ with $a \le 0$:
\begin{gather}
		-(\eps^* w')' - b^* w' + c^* w = f^* \quad \text{in $(a,1)$},\qquad w(a) = 0,\; w(1) = u_1,
		\label{eq:prob_widen}
\end{gather}
with smooth functions $\eps^*$, $b^*$, $c^*$ and $f^*$ defined on $(a,1)$ and satisfying
\begin{gather}
	\begin{aligned}
		C \underline \eps &\le \eps^*(x) \le \overline \eps\\
		\beta &\le b^*(x)\\		
		0 &\le c^*(x)
	\end{aligned} \quad \text{for $x \in [a,1]$}.
	\label{eq:bandepsstar}
\end{gather}

\begin{lem}\label{lem:absch_u}
	Suppose $0 \le (\eps^*)'$ on $[a,1]$. Then the solution $w$ of problem \eqref{eq:prob_widen} satisfies
	\begin{gather}		
		|w(x)| \le C, \quad \text{$x \in [a,1]$}.\label{eq:uestim}
	\end{gather}	
\end{lem}
\begin{proof}
	Using the comparison principle induced by Lemma \ref{lem:maxprinciple} with the barrier functions $\psi^\pm$ defined by
	\begin{gather*}
		\psi^\pm(x) \coloneqq \pm \frac{1}{\beta} \|f^*\|_\infty (1-x) \pm |u_1|
	\end{gather*}
	one obtains the result, because
	\begin{align*}
		\left(\mathcal{L}_\eps \psi^+ \right)(x) &= \frac{b^*(x)+(\eps^*)'(x)}{\beta} \|f^*\|_\infty + c^*(x) \left( \frac{1}{\beta} (1-x) + |u_1| \right) \ge \|f^*\|_\infty \ge \left(\mathcal{L}_\eps w \right)(x) \quad \text{in $(a,1)$},\\
		\psi^+(a) &= \frac{1}{\beta} \|f^*\|_\infty (1-a) + |u_1| \ge 0 = w(a),\\
		\psi^+(1) &= |u_1| \ge u_1 = w(1).\\
	\end{align*}
	Hence $w \le \psi^+ \le C$ on $[a,1]$. The other bound follows similarly with $\psi^-$.
\end{proof}

The following argument is an extension of \cite{KT78}.

\begin{lem}\label{lem:absch_u1}
	Suppose $0 \le (\eps^*)'$ on $[a,1]$ and set $e_a(t) \coloneqq \int_a^t 1/\eps^*(z) \D z$. Then the solution $w$ of problem \eqref{eq:prob_widen} satisfies
	\begin{gather}		
		|w'(x)| \le C \left(1 + \frac{1}{\eps^*(x)} \E^{-\beta e_a(x)} \right), \quad \text{$x \in [a,1]$}. \label{eq:u1estim}
	\end{gather}	
\end{lem}
\begin{proof}
	For the sake of readability, we drop the star from the notation of the functions $\eps^*$, $b^*$, $c^*$ and $f^*$ within this proof.
	Set $h \coloneqq f - c w$. The problem
	\begin{gather*}
		-w''(x) - \frac{b(x)+\eps'(x)}{\eps(x)} w'(x) = \frac{h(x)}{\eps(x)},\qquad w(a) = 0,\; w(1) = u_1
	\end{gather*}
	is equivalent to problem \eqref{eq:prob_widen}. It's solution $w$ admits the representation
	\begin{gather*}
		w(x) = w_p(x) + K_1 + K_2 \int_a^x \E^{-\big( B(t) - B(a)\big)} \D t,
	\end{gather*}
	where
	\begin{align*}
		w_p(x) &\coloneqq -\int_a^x z(t) \D t,\qquad z(x) \coloneqq \int_a^x \frac{h(t)}{\eps(t)} \E^{-\big(B(x)-B(t)\big)} \D t,\\
		B(x) &\coloneqq \int_a^x \frac{b(t) + \eps'(t)}{\eps(t)} \D t = \int_a^x \frac{b(t)}{\eps(t)} \D t + \ln \big(\eps(x) \big) - \ln \big(\eps(a) \big),
	\end{align*}
	i.e.~$B$ is an indefinite integral of $(b+\eps')/\eps$. The constants $K_1$ and $K_2$ may depend on $\eps$.
	The boundary condition $w(a) = 0$ yields $K_1 = 0$ whereas the other boundary condition $w(1) = u_1$ gives
	\begin{gather}
		\begin{split}
			u_1-w_p(1) &= K_2 \int_a^1 \E^{-\big( B(t) - B(a)\big)} \D t = K_2 \int_a^1 \E^{-\int_a^t \frac{b(z)}{\eps(z)} \D z + \ln \left( \frac{\eps(a)}{\eps(t)} \right)} \D t \\
			&= K_2\, \eps(a) \int_a^1 \frac{1}{\eps(t)} \E^{-\int_a^t \frac{b(z)}{\eps(z)} \D z} \D t.
		\end{split}
		\label{eq:boundarycond1}
	\end{gather}
	Because of Lemma \ref{lem:absch_u} we know $\|w\|_{\infty} \le C$. Thus
	\begin{gather}
		|z(x)| \le C \int_a^x \frac{1}{\eps(t)} \E^{-\big(B(x)-B(t)\big)} \D t.
		\label{eq:zestim}
	\end{gather}
	For $t \le x$ a simple calculation yields
	\begin{align*}
		\frac{1}{\eps(t)} \E^{-\big(B(x)-B(t)\big)} &= \E^{-B(x) +B(t) - \ln \big( \eps(t) \big)}
		= \E^{-\int_a^x \frac{b(z)}{\eps(z)} \D z + \int_a^t \frac{b(z)}{\eps(z)} \D z - \ln\big( \eps(x) \big)}\\
 		&= \frac{1}{\eps(x)} \E^{-\int_t^x \frac{b(z)}{\eps(z)} \D z} \le \frac{1}{\eps(x)} \E^{- \beta \int_t^x \frac{1}{\eps(z)} \D z} = \frac{1}{\eps(x)} \E^{- \beta \big( e_a(x) - e_a(t)\big)}.
	\end{align*}
	Inserting this estimate into \eqref{eq:zestim} and applying Lemma \ref{lem:intestim} with $\ell = 0$ we obtain
	\begin{align*}
		|z(x)| &\le \frac{C}{\eps(x)} \int_a^x \E^{- \beta \big( e_a(x) - e_a(t)\big)} \D t \le \frac{C}{\eps(x)} \E^{-\beta e_a(x)} \int_a^x \E^{\beta e_a(t)} \D t\\
		&\le \frac{C}{\eps(x)} \E^{-\beta e_a(x)} \frac{1}{\beta + \sigma_0} \eps(x) \E^{\beta e_a(x)} \le C.
	\end{align*}
	Moreover $|z(x)| \le C$ for all $x \in [0,1]$ implies $|w_p(1)| \le C$.
	We still need to estimate
	\begin{align*}
		\begin{split}
			\int_a^1 \frac{1}{\eps(t)} \E^{-\int_a^t \frac{b(z)}{\eps(z)} \D z} \D t &\ge \int_a^1 \frac{1}{\eps(t)} \E^{-\|b\|_\infty e_a(t)} \D t = \int_0^{e_a(1)} \E^{-\|b\|_\infty s} \D s\\
			&= \frac{1}{\|b\|_\infty} \left( 1 - \E^{-\|b\|_\infty e_a(1)} \right) \ge C.
		\end{split}
	\end{align*}
	Here we used the substitution $s = e_a(t)$ with $\D s/\D t = e_a'(t) = 1/\eps(t)$.
	Thus, with \eqref{eq:boundarycond1} we get
	\begin{gather*}
		|K_2| \le C \frac{1}{\eps(a)}.
	\end{gather*}
	Combining this with
	\begin{gather}
		w'(x) = -z(x) + K_2 \E^{-\big( B(x) - B(a)\big)},
		\label{eq:uprime}
	\end{gather}
	we obtain
	\begin{gather*}
		|w'(x)| \le |z(x)| + C \frac{1}{\eps(x)} \E^{-\beta e_a(x)} \le C \left(1 + \frac{1}{\eps(x)} \E^{-\beta e_a(x)} \right)
	\end{gather*}
	and \eqref{eq:u1estim} is verified.
\end{proof}

\begin{rem}
	Again in the classical case where $\eps^*$ is constant the formula \eqref{eq:u1estim} reduces to \eqref{eq:classic} which is known to be optimal.
\end{rem}

\begin{rem}
	An inspection of the proof of \eqref{eq:u1estim} shows that the assumption $\overline \eps \ll 1$ can be dropped provided $w$ remains uniformly bounded and $e_a(1)$ is sufficiently large --- Remark that $e_a$ is a strictly increasing function.
\end{rem}

	With \eqref{eq:u1estim} we readily obtain a pointwise estimate for $w''$.
	
\begin{lem}
	 Let $0 \le (\eps^*)'$ on $[a,1]$ and set $e_a(t) \coloneqq \int_a^t 1/\eps^*(z) \D z$. Then the solution $w$ of problem \eqref{eq:prob_widen} satisfies
	 \begin{gather}
		|w''(x)| \le C \frac{1+(\eps^*)'(x)}{\eps^*(x)} \left(1 + \frac{1}{\eps^*(x)} \E^{-\beta e_a(x)} \right),  \quad \text{$x \in [a,1]$}. \label{eq:u2bound}
	 \end{gather}
\end{lem}
\begin{proof}
	The result is an immediate consequence of \eqref{eq:prob_widen}, $\|w\|_\infty \le C$ and \eqref{eq:u1estim}.
\end{proof}

\begin{lem}\label{lem:u2estim}
	Suppose $0 \le (\eps^*)'$ on $[a,1]$. Then for the solution $w$ of Problem \eqref{eq:prob_widen}
		\begin{gather}		
			|w''(x)| \le C \left(1 + K_4(x) + \frac{1}{\eps^*(x)^2} \left( 1 + (\eps^*)'(a) + \|(\eps^*)''\|_{L^1(a,x)} \right) \E^{-\beta e_a(x)} \right) \label{eq:u2estim},  \quad \text{$x \in [a,1]$}
		\end{gather}	
		holds with $e_a(t) \coloneqq \int_a^t 1/\eps^*(z) \D z$ and $K_4(x) \le C \min \{\|(\eps^*)''\|_{\infty,(a,x)}, \frac{1}{\sqrt{\eps^*(x)}}\|(\eps^*)''\|_{0,(a,x)}\}$.
\end{lem}
\begin{proof}In order to simplify the illustration we again drop the star from the notation of the functions $\eps^*$, $b^*$, $c^*$ and $f^*$ within this proof.
	A differentiation of \eqref{eq:prob_widen} yields
	\begin{gather*}
		-w^{(3)} - \frac{b + 2 \eps'}{\eps} w'' = \frac{f' + (b'+\eps'' - c)w' - c'w}{\eps} \eqqcolon g.
	\end{gather*}
	Thus, we obtain a differential equations for $\omega \coloneqq w''$, indeed $-\omega' -(b+2\eps')/\eps \, \omega = g$.
	Setting
	\begin{gather*}
		\tilde B(x) \coloneqq \int_a^x \frac{b(t) + 2 \eps'(t)}{\eps(t)} \D t = \int_a^x \frac{b(t)}{\eps(t)} \D t + 2 \ln \big( \eps(x) \big) - 2 \ln \big( \eps(a) \big)
	\end{gather*}
	(i.e.~$\tilde B$ is an indefinite integral of $(b+2\eps')/\eps$) the function $\omega$ can be represented as
	\begin{gather}
		\omega(x) = K_3 \E^{-\big(\tilde B(x) - \tilde B(a) \big)} - \int_a^x g(t) \E^{-\big(\tilde B(x) - \tilde B(t) \big)} \D t.
		\label{eq:wdgl}
	\end{gather}
	Here the constant $K_3$ may depend on $\eps$. Because of the identity
	\begin{gather*}
		\E^{-\big(\tilde B(x) - \tilde B(t) \big)} = \left( \frac{\eps(t)}{\eps(x)} \right)^2 \E^{-\int_t^x \frac{b(z)}{\eps(z)}\D z}
	\end{gather*}
	and $K_3 = \omega(a) = w''(a)$ the representation \eqref{eq:wdgl} implies
	\begin{gather}
		|\omega(x)| \le |w''(a)| \left( \frac{\eps(a)}{\eps(x)} \right)^2 \E^{- \beta e_a(x)} + \int_a^x |g(t)| \left( \frac{\eps(t)}{\eps(x)} \right)^2 \E^{-\beta \big( e_a(x) - e_a(t) \big)} \D t.
		\label{eq:estimw}
	\end{gather}
	Because $|g(t)| \le \frac{|b'(t)+\eps''(t) - c(t)|}{\eps(t)}|w'(t)| + \frac{|c'(t)| }{\eps(t)} |w(t)| + \frac{|f'(t)|}{\eps(t)}$ the integral in \eqref{eq:estimw} is dominated by the sum of the two integrals $I_0(x)$ and $I_1(x)$ with
	\begin{align*}
		I_0(x) &\coloneqq \int_a^x \left( \frac{|f'(t)|}{\eps(t)} + \frac{|c'(t)| }{\eps(t)} |w(t)|\right) \left( \frac{\eps(t)}{\eps(x)} \right)^2 \E^{-\beta \big( e_a(x) - e_a(t) \big)} \D t,\\
		I_1(x) &\coloneqq \int_a^x \frac{|b'(t)+\eps''(t) - c(t)|}{\eps(t)}|w'(t)| \left( \frac{\eps(t)}{\eps(x)} \right)^2 \E^{-\beta \big( e_a(x) - e_a(t) \big)} \D t.
	\end{align*}
	Using $\|w\|_\infty \le C$ and applying Lemma \ref{lem:intestim} with $\ell = 1$ we see that
	\begin{gather}
		I_0(x) \le C \frac{1}{\eps(x)^2} \E^{-\beta e_a(x)} \int_a^x \eps(t) \E^{\beta e_a(t)} \D t \le C \frac{1}{\eps(x)^2} \E^{-\beta e_a(x)} \eps(x)^2 \E^{\beta e_a(x)} \le C.
		\label{eq:I0}
	\end{gather}
	For $I_1$ the bound \eqref{eq:u1estim} yields with Lemma \ref{lem:intestim} ($\ell = 1$)
	\begin{gather}
		\begin{split}
			I_1(x) &\le C \frac{1}{\eps(x)^2} \E^{-\beta e_a(x)} \int_a^x \big(1 + |\eps''(t)|\big)\, \left(1 + \frac{1}{ \eps(t)} \E^{-\beta e_a(t)} \right)\, \eps(t) \E^{\beta e_a(t)} \D t\\
			&\le C \frac{1}{\eps(x)^2} \E^{-\beta e_a(x)} \int_a^x \eps(t) \E^{\beta e_a(t)} + |\eps''(t)| \eps(t) \E^{\beta e_a(t)} + 1 + |\eps''(t)| \D t\\
			&\le C \frac{1}{\eps(x)^2} \E^{-\beta e_a(x)} \left( \eps(x)^2 \E^{\beta e_a(x)} + K_4(x) + \left(1 + \|\eps''\|_{L^1(a,x)}\right) \right)\\
			&\le C\left( 1 + K_4(x) + \frac{1}{\eps(x)^2} \left(1 + \|\eps''\|_{L^1(a,x)} \right) \E^{-\beta e_a(x)} \right)
		\end{split}
		\label{eq:I1}
	\end{gather}
	with $K_4(x) \coloneqq \int_a^x |\eps''(t)| \eps(t) \E^{\beta e_a(t)} \D t$. H\"older's inequality and the Cauchy-Schwarz inequality yield for $K_4(x)$ with Lemma \ref{lem:intestim}
	\begin{subequations}
		\begin{align}
			K_4(x) &\le C \frac{1}{\eps(x)^2} \E^{-\beta e_a(x)} \|\eps''\|_{\infty,(a,x)} \int_a^x \eps(t) \E^{\beta e_a(t)} \D t \le C \|\eps''\|_{\infty,(a,x)}\\
			K_4(x) &\le C \frac{1}{\eps(x)^2} \E^{-\beta e_a(x)} \|\eps''\|_{0,(a,x)} \left(\int_a^x \eps(t)^2 \E^{2\beta e_a(t)} \D t\right)^{1/2} \le C \frac{1}{\sqrt{\eps(x)}}\|\eps''\|_{0,(a,x)}
		\end{align}
		\label{eq:K4all}
	\end{subequations}
	From \eqref{eq:u2bound} we deduce the bound
	\begin{gather}
		|w''(a)| \le C \frac{1 + \eps'(a)}{\eps(a)^2}.
		\label{eq:upp0}
	\end{gather}
	We conclude our proposition by collecting \eqref{eq:estimw}, \eqref{eq:I0}, \eqref{eq:I1}, \eqref{eq:K4all} and \eqref{eq:upp0}.
\end{proof}

\begin{thm}[Solution Decomposition]\label{thm:soldec}
	Suppose $0 \le \eps'$ on $[0,1]$ and define $e(t) \coloneqq \int_0^t 1/\eps(z) \D z$. Then there exists a constant $S_0$ with $|S_0| \le C$ such that the solution $u$ of \eqref{eq41} can be decomposed into the sum of a smooth part $S$ and an exponential boundary layer component $E^{BL}$, i.e.~$u = S + E^{BL}$ such that $S$ and $E$ solve the boundary-value problems
	\begin{subequations}
		\begin{alignat}{3}
			\mathcal{L}_\eps S &= f\quad \text{in $(0,1)$},\qquad& S(0) &= S_0,\quad& S(1) &= 0, \label{eq:bvpS}\\
			\mathcal{L}_\eps E^{BL} &= 0\quad \text{in $(0,1)$},\qquad& E^{BL}(0) &= - S_0,\quad& E^{BL}(1) &= 0.\label{eq:bvpE}
		\end{alignat}	
		Moreover there exists a constant $C$ such that for $x \in [0,1]$
		\begin{align}
			|S^{(k)}(x)| & \le C\quad \text{for $k=0,1$}, \label{eq:boundS01}\\
			|S''(x)| & \le C \frac{1+\eps'(x)}{\eps(x)}, \label{eq:boundS2}\\
			|\eps(x) S''(x) + \eps'(x) S'(x)| & \le C, \label{eq:boundS2n}\\
			\shortintertext{and}
			\big |(E^{BL})^{(k)}(x) \big | & \le C \frac{1}{\eps(x)^k} \E^{-\beta e(x)} \quad \text{for $k=0,1$}, \label{eq:boundE01}\\
			\big |(E^{BL})''(x) \big | &\le C \frac{1 + \eps'(x)}{\eps(x)^2} \E^{-\beta e(x)}. \label{eq:boundE2}
		\end{align}
	\end{subequations}
\end{thm}
\begin{proof}
	We start off with the regular solution component $S$:
	Fix $a < -\frac{1}{\beta} \underline \eps \ln \frac{1}{\underline \eps} < 0$. On the the interval $(a,1)$ choose smooth extension $\eps^*$, $b^*$, $c^*$ and $f^*$ of $\eps$, $b$, $c$ and $f$ in such a way that $\eps^*$ is non-decreasing and the assumptions \eqref{eq:bandepsstar} are met.
Thus we can apply Lemma \ref{lem:absch_u} and Lemma \ref{lem:absch_u1} to the boundary value problem
\begin{gather*}
	-\big(\eps^* (S^*)'\big)' - b^* (S^*)' + c^* S^* = f^* \quad \text{in $(a,1)$},\qquad S^*(a) = 0,\quad S^*(1) = u_1
\end{gather*}
to obtain
\begin{align}
|S^*(x)| &\le C\qquad \text{and}\qquad |(S^*)'(x)| \le C \left(1 + \frac{1}{\eps^*(x)} \E^{-\beta e_a(x)} \right)\quad \text{for $x \in [a,1]$}. \label{eq:Sbounds}
\end{align}
Since $\eps$ is non-decreasing we can set $\underline \eps \coloneqq \eps(0) = \eps^*(0)$.
This implies for all $x \in (a,0]$ that $C \underline \eps \le \eps^*(x) \le \underline \eps$. Hence
\begin{gather*}
	\E^{-\beta e_a(0)} = \E^{-\beta \int_a^0 \frac{1}{\eps^*(z)} \D z} \le \E^{ \frac{\beta a}{\underline \eps}} < \underline \eps.
\end{gather*}
The fact that $e_a$ is a strictly increasing function implies $e_a(x) > e_a(0)$. Thus we arrive at
\begin{gather*}
	\E^{-\beta e_a(x)} \le \E^{-\beta e_a(0)} < \underline \eps \quad \text{for $x \in [0,1)$}.
\end{gather*}
From \eqref{eq:Sbounds} it now follows that
\begin{gather*}
	\big|(S^*)^{(k)}(x)\big| \le C\quad \text{for $k=0,1$ and $x \in [0,1]$}
\end{gather*}
since $\eps^*|_{(0,1)} = \eps \ge \underline \eps$. Setting $S \coloneqq S^*|_{(0,1)}$ it satisfies the boundary value problem \eqref{eq:bvpS} because $b^*|_{(0,1)} = b$, $c^*|_{(0,1)} = c$ as well as $f^*|_{(0,1)} = f$. The bound on $S^*$ and $(S^*)'$ yields \eqref{eq:boundS01}, in particular $S(0) = S_0 \coloneqq S^*(0)$ with $|S_0| \le C$. Since \eqref{eq:boundS2} and  \eqref{eq:boundS2n} are immediate consequences of \eqref{eq:bvpS} and \eqref{eq:boundS01} all propositions for the regular part $S$ are verified.

To bound the layer component $E^{BL}$ we use the barrier functions $\phi^\pm$ defined by
\begin{gather*}
	\phi^\pm (x) = \pm \big|E^{BL}(0)\big| \E^{-\beta e(x)}.
\end{gather*}
A simple calculation yields
\begin{align*}
	\big( \mathcal{L}_\eps \phi^+ \big)(x) &= \big|E^{BL}(0)\big| \left( \beta \frac{b(x) + \eps'(x) -\beta - \eps'(x)}{\eps(x)} + c(x) \right) \E^{-\beta e(x)} \ge 0 = \big( \mathcal{L}_\eps E^{BL} \big)(x),\\
	\phi^+(0) &= \big|E^{BL}(0)\big| \ge E^{BL}(0),\\
	\phi^+(1) &= \big|E^{BL}(0)\big| \E^{-\beta E(1)} \ge 0 = E^{BL}(1).\\
\end{align*}
Hence $E^{BL}(x) \le \phi^+(x) \le C \E^{-\beta e(x)}$. Because $\phi^- = -\phi^+$ the estimate \eqref{eq:boundE01} for $k=0$ follows.

For the first derivative of the boundary layer term $E^{BL}$ we use the representation
\begin{align*}
	E^{BL} (x) &= \int_x^1 z(s) \D s + K \int_x^1 \E^{-B(s)} \D s
	\shortintertext{with}
	z(x) &= -\int_0^x \frac{b(s)+\eps'(s)}{\eps(s)} E^{BL} (s) \E^{-\big(B(x)-B(s)\big)} \D s,\\
	B(x) &= \int_0^x \frac{b(s)+\eps'(s)}{\eps(s)} \D s = \int_0^x \frac{b(s)}{\eps(s)} \D s + \ln\big(\eps(x)\big) - \ln\big(\eps(0)\big).
\end{align*}
The estimate \eqref{eq:boundE01} with $k=0$ yields for $z$:
\begin{align*}
	|z(x)| &\le C \int_0^x \frac{b(s)+\eps'(s)}{\eps(s)} \E^{-\beta E(s)} \frac{\eps(s)}{\eps(x)} \E^{-\beta \big( e(x) - e(s) \big)} \D s\\
	& \le C \frac{1}{\eps(x)} \E^{-\beta e(x)} \int_0^x \big( 1+ \eps'(s) \big) \D s \le C \frac{1}{\eps(x)} \E^{-\beta e(x)}.
\end{align*}
The constant $K$ is governed by the boundary condition $E^{BL}(0) = u_0 - S_0$:
\begin{gather}
	K \int_0^1 \E^{-B(s)} \D s = u_0 - S_0 - \int_0^1 z(s) \D s.
	\label{eq:Kboundary}
\end{gather}
Using the substitution $t = e(s)$ with $\frac{\D t}{\D s} = e'(s) = \frac{1}{\eps(s)}$ we obtain
\begin{align*}
	\int_0^1 \E^{-B(s)} \D s &= \int_0^1 \frac{\eps(0)}{\eps(s)} \E^{-\int_0^s \frac{b(x)}{\eps(x)} \D x} \D s \ge \int_0^1 \frac{\eps(0)}{\eps(s)} \E^{-\|b\|_\infty e(s)} \D s = \eps(0) \int_0^{e(1)} \E^{-\|b\|_\infty t} \D t \ge C \eps(0)
\shortintertext{and}
	\int_0^1 z(s) \D s &\le \int_0^1 |z(s)| \D s \le C \int_0^1 \frac{1}{\eps(s)} \E^{-\beta e(s)} \D s \le C \int_0^{e(1)} \E^{-\beta t} \D t \le C \int_0^\infty \E^{-\beta t} \D t \le C.
\end{align*}
Hence \eqref{eq:Kboundary} gives $|K| \le \frac{C}{\eps(0)}$ and because
\begin{gather*}
	\big(E^{BL}\big)'(x) = -z(x) - K \E^{-B(x)}
\end{gather*}
we can estimate
\begin{gather*}
	\left| \big(E^{BL}\big)'(x) \right| \le |z(x)| + |K| \E^{-B(x)} \le C \frac{1}{\eps(x)} \E^{-\beta e(x)} + \frac{C}{\eps(0)} \frac{\eps(0)}{\eps(x)} \E^{-\beta e(x)}
\end{gather*}
which is \eqref{eq:boundE01} for $k=1$. For the remaining result \eqref{eq:boundE2} we use the differential equation \eqref{eq:bvpE} and the bounds \eqref{eq:boundE01}:
\begin{gather*}
	\left| \big(E^{BL}\big)''(x) \right| \le \frac{|b(x)|+\eps'(x)}{\eps(x)} \left| \big(E^{BL}\big)'(x) \right| + \frac{|c(x)|}{\eps(x)} \left| \big(E^{BL}\big)(x) \right| \le C \frac{1+ \eps'(x)}{\eps(x)^2} \E^{-\beta e(x)}.
\end{gather*}
\end{proof}

\begin{rem}
	It is possible to use Lemma \ref{lem:u2estim} to prove the existence of a solution decomposition
$u=S+E$ with the better bounds \eqref{eq:boundS01*} and \eqref{eq:boundE01*} for the derivatives of $S$
 and $E$. This requires the additional assumption that $\eps'$ is bounded, moreover additional assumptions
 concerning smooth extensions of $\eps''$.
\end{rem}

\end{section}

\end{document}